\documentclass[twoside]{amsart}

\usepackage{epsfig,comment}
\usepackage{mathtools}
\usepackage[subnum]{cases}
\usepackage{enumerate}
\usepackage{bbm,bm}
\usepackage{amsmath,amssymb,mathrsfs}
\usepackage{color}
\usepackage{graphicx} 
\usepackage{subcaption}

\newtheorem{theorem}{Theorem}
\newtheorem{definition}{Definition}
\newtheorem{lemma}{Lemma}
\newtheorem{corollary}{Corollary}
\newtheorem{remark}{Remark}

\newtheorem{proposition}{Proposition}

\newtheorem{problem}{Problem}

\renewcommand{\Re}{\mathbb{R}}

\DeclareMathOperator{\conv}{conv}

\DeclareMathOperator{\perim}{perim}
\DeclareMathOperator{\bd}{bd}
\DeclareMathOperator{\Sph}{\mathbb{S}}

\DeclareMathOperator{\inter}{int}

\DeclareMathOperator{\arclength}{arclength}

\DeclareMathOperator{\MinMin}{MinMin}
\DeclareMathOperator{\MinMax}{MinMax}
\DeclareMathOperator{\MaxMin}{MaxMin}
\DeclareMathOperator{\MaxMax}{MaxMax}

\begin{document}

\title[Curves with increasing chords]{Curves with increasing chords in normed planes}

\author[Z. L\'angi]{Zsolt L\'angi}
\author[S. Lengyel]{S\'ara Lengyel}

\address{Zsolt L\'angi, Bolyai Institute, University of Szeged,\\
Aradi V\'ertan\'uk tere 1, 6720 Szeged, Hungary, and\\
HUN-REN Alfr\'ed R\'enyi Institute of Mathematics,\\
Re\'altanoda utca 13-15, H-1053, Budapest, Hungary}
\email{zlangi@server.math.u-szeged.hu}
\address{S\'ara Lengyel, Budapest University of Technology and Economics,\\
M\H uegyetem rkp. 3., H-1111 Budapest, Hungary}
\email{saralengyel03@gmail.com}

\thanks{Partially supported by the ERC Advanced Grant ``ERMiD'' and the National Research, Development and Innovation Office, NKFI, K-147544 grant}


\begin{abstract}
A curve has the increasing chord property if for any points $a,b,c,d$ in this order on the curve, the distance of $a,d$ is not smaller than that of $b,c$. 
Answering a conjecture of Larman and McMullen, Rote proved in 1994 that the arclength of a curve in the Euclidean plane with the increasing chord property is at most $\frac{2\pi}{3}$ times the distance of its endpoints, and this inequality is sharp. In this note we generalize the result of Rote for curves in a normed plane with a strictly convex norm, based on an investigation of the geometric properties of involutes in normed planes. We also discuss some related extremum problems.
\end{abstract}

\maketitle

\section{Introduction}

A curve satisfies the increasing chord property if for any points $a,b,c,d$ in this order on the curve, the distance of $a,d$ is not smaller than that of $b,c$.
In 1971, Binmore \cite{Bin71} asked if there is a universal upper bound $C > 0$ such that for any curve $\Gamma$ in the Euclidean plane satisfying the increasing chord property, the length of $\Gamma$ is at most $C$ times the distance of its endpoints. This question was answered by Larman and McMullen \cite{LM72}, who proved that the statement is true with $C=2\sqrt{3}$, while the best constant $C= \frac{2\pi}{3}$ was determined by Rote \cite{Rot94} in 1994; here equality is attained e.g. if $\Gamma$ is the union of two sides of a Reuleaux triangle. For more information on the early results of this problem, the interested reader is referred to \cite[Problem G3]{CFG91} of the open problem book of Croft, Falconer and Guy.

Concepts related to the increasing chord property are also extensively studied in the literature. In particular, the stretch factor of a curve $\gamma$ is defined as the ratio of its arclength to the distance of its endpoints, and the geometric dilation of $\gamma$ is the maximum stretch factor of all sub-arcs of $\gamma$. Dilation is also defined for a graph embedded in Euclidean space as the maximum stretch factor of its edges. These concepts are studied e.g. in \cite{AKK+08, CDMT21,DEK+07, NS00}. Algorithms to check whether a polygonal arc in $d$-dimensional Euclidean space with $d \leq 3$ satisfies the increasing chord property can be found in \cite{ACG+13}.

Even though the notion of increasing chords can be naturally extended to curves in any metric space, this property as well as all related concepts mentioned in the previous paragraph were studied only in Euclidean space. The aim of this note is to extend this investigation, in particular the question of Binmore, to curves in normed spaces.

Before stating our main result, we recall the well-known fact that the unit ball of a normed space $\mathbb{M}$ is an origin-symmetric convex body $M$, and vice versa, every such convex body $M$ induces a unique normed space $\mathbb{M}$. Based on this, we denote the family of origin-symmetric $d$-dimensional convex bodies by $\mathcal{M}_d$, and for every $M \in \mathcal{M}_d$, the induced normed space by $\mathbb{M}$. For any two points $p,q \in \mathbb{M}$, we denote their distance by $||p-q||_M$.

\begin{theorem}\label{thm:main}
Let $p,q$ be points in a normed plane $\mathbb{M}$ with a strictly convex unit disk $M$, at unit distance apart. Let $f:[0,1] \to \mathbb{M}$ be a continuous curve satisfying the increasing chord property such that $f(0)=p$ and $f(1)=q$. Then the arclength of $f$, in the norm of $\mathcal{M}$, is at most half of the perimeter of $(p+M) \cap (q+M)$.
\end{theorem}

We remark that a Reuleaux triangle in a normed plane $\mathbb{M}$ with unit disk $M$ is the intersection of three translates of $ \lambda M$ for some $\lambda > 0$ centered at points at pairwise distances $\lambda$, measured in the norm (see e.g. \cite{C66, MM07, L16}). Thus, Theorem~\ref{thm:main} is a straightforward generalization of the result of Rote \cite{Rot94} for strictly convex norms. The proof of Theorem~\ref{thm:main} follows the proof of Theorem 1 of Rote in \cite{Rot94}, with the necessary modifications, and is based on the properties of involutes in normed planes.

Base on Theorem~\ref{thm:main}, for any points $p,q \in \mathbb{M}$ with $M \in \mathcal{M}_2$ and satisfying $||p-q||_M=1$, let us define $L_M(p,q)$ as half of the perimeter of $(p+M) \cap (q+M)$, measured in the norm of $\mathbb{M}$. Observe that, unlike in the Euclidean case, this quantity depends not only on the norm, but also on the direction of the segment $[p,q]$. This leads us to the following notion.

\begin{definition}\label{defn:extremalvalues}
Set
\begin{eqnarray*}
\MinMin(L) & = & \min \{ \min \{ L_M(p,q) : p,q \in \mathbb{M}, ||q-p||_M=1 \} : M \in \mathcal{M}_2 \},\\
\MinMax(L) & = & \min \{ \max \{ L_M(p,q) : p,q \in \mathbb{M}, ||q-p||_M=1 \} : M \in \mathcal{M}_2 \},\\
\MaxMin(L) & = & \max \{ \min \{ L_M(p,q) : p,q \in \mathbb{M}, ||q-p||_M=1 \} : M \in \mathcal{M}_2 \},\\
\MaxMax(L) & = & \max \{ \max \{ L_M(p,q) : p,q \in \mathbb{M}, ||q-p||_M=1 \} : M \in \mathcal{M}_2 \}.
\end{eqnarray*}
\end{definition}

Our second result is the following.

\begin{theorem}\label{thm:extremals}
We have
\begin{enumerate}
\item[(i)] $\MinMin(L)=2$, with equality e.g. if $M$ is a parallelogram or an affinely regular hexagon;
\item[(ii)] $\MinMax(L)=2$, with equality e.g. if $M$ is an affinely regular hexagon;
\item[(iii)] $\frac{2\pi}{3} \leq \MaxMin(L) \leq \frac{8}{3}$;
\item[(iv)] $\MaxMax(L)=3$, with equality e.g. if $M$ is a parallelogram.
\end{enumerate}
\end{theorem}

Throughout the paper, for points $p,q$ in a normed plane $\mathbb{B}$, we denote by $[p,q]$ the closed segment with endpoints $p,q$. If $X \subset \mathbb{M}$, its convex hull is denoted by $\conv(X)$. For a continuous curve $\Gamma : [0,1] \to \mathbb{M}$, we define the arclength $\arclength_M(\Gamma)$ of $\Gamma$ in the natural way, namely as the supremum of the total normed length $\sum_{i=1}^n ||\Gamma(t_i)-\Gamma(t_{i-1})||_M$ of any polygonal curve induced by a subdivision $0=t_0 < t_1 < t_2 < \ldots < t_n=1$ of $[0,1]$. The arclength of a convex disk $K$ (i.e. a compact, planar, convex set with nonempty interior) in $\mathbb{M}$ is called the \emph{perimeter} of $K$, denoted as $\perim_M(K)$.
For simplicity, as it is often done, we also equip $\mathbb{M}$ with an underlying Euclidean plane, induced by a Descartes coordinate system.

The structure of the paper is as follows. In Section~\ref{sec:involutes}, as preliminary material to the proofs of Theorems~\ref{thm:main} and \ref{thm:extremals}, we introduce a geometric definition of the involutes of a plane convex body in a normed plane $\mathbb{M}$, and investigate their properties. We remark that involutes in normed planes were recently investigated in \cite{BMS19} from a differential geometry point of view, nevertheless our approach is more general as it does not make conditions on the differentiability of the curve.
Then, in Section~\ref{sec:main} we prove Theorem~\ref{thm:main}, and in Section~\ref{sec:extremals} we prove Theorem~\ref{thm:extremals}. Finally, in Section~\ref{sec:remarks} we collect some additional remarks and questions.

\section{Involutes in normed planes}\label{sec:involutes}

Involutes of convex disks are well-known curves in the theory of the differential geometry of curves in the Euclidean plane (see e.g. \cite{R00}). These concepts were generalized for normed planes for smooth curves e.g. in \cite{ABS00, BMS19}, based on the notion of curvature in normed planes. In addition, they were investigated for polygonal norms and for constant width bodies in this norm in \cite{CM16}. The goal of this section is to introduce a geometric definition that can be applied for any norm and convex disk, and investigate the properties of these curves. Our definition, applied for smooth curves, coincides with the one in the literature (see e.g. \cite{ABS00}).

\begin{definition}\label{defn:involute}
Let $C$ be a convex disk in the normed plane $\mathbb{M}$, and let $p \in \bd(C)$. We orient $\bd(C)$ in counterclockwise direction, and orient all tangent (i.e. supporting) lines of $C$ accordingly. We assume that the unit vector $(1,0)=(\cos 0, \sin 0)$ is an oriented unit tangent vector of a tangent line of $C$ at $p$. For any $\theta \in [0,2\pi)$, we denote the oriented tangent line with tangent vector $(\cos \theta, \sin \theta)$ by $L_{\theta}$, and for any $q \in \bd(C)$ we set $d(p,q)$ as the $M$-length of the arc of $\bd(C)$ from $p$ to $q$ according to this orientation. Then we define $\Gamma_p$ as follows:
$\Gamma_p(\theta)$ is the point of $L_{\theta}$ whose signed $M$-distance from $q$ is equal to $-d(p,q)$ (see Figure~\ref{fig:involute}).
\end{definition}

\begin{figure}[ht]
\begin{center}
\includegraphics[width=\textwidth]{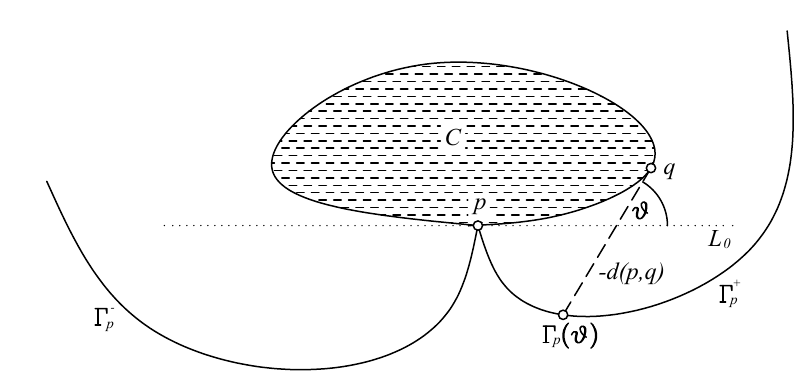}
\caption{An illustration for the definition of an involute of a convex disk.}
  \label{fig:involute}
\end{center}
\end{figure}

\begin{remark}
The definition of $\Gamma_p(\theta)$ is independent of $q$, i.e. if $L_{\theta}$ is a segment $S$, then for any $q \in S$ the point $\Gamma_p(\theta)$ is the same point of $L_{\theta}$.
\end{remark}

\begin{remark}
The curve $\Gamma_p$ of Definition~\ref{defn:involute}, defined for $\theta \in [0,2\pi)$ can be naturally extended to the interval $(-\infty,\infty)$. This curve is the union of two `spirals', each starting at $p$. We call the component defined on $[0,\infty)$ the \emph{positive branch} of $\Gamma_p$, denoted as $\Gamma_p^+$, and the one defined on $(-\infty,0]$ the \emph{negative branch} of $\Gamma_p$, denoted as $\Gamma_p^-$.
\end{remark}

Recall that a vector $v$ is \emph{Birkhoff-James orthogonal} (or simply \emph{Birkhoff orthogonal}) to $u$ if $||v||_M \leq ||v+t u||_M$ for all $t \in \Re$ \cite{AMW22}. In this paper, if $L$ is a line with tangent vector $v$, and $L'$ is a line with tangent vector $u$, we say that $L$ is Birkhoff orthogonal to $L'$ if $v$ is Birkhoff orthogonal to $u$.

\begin{lemma}\label{lem:convexity}
The curve $\Gamma_p$ satisfies the following properties:
\begin{enumerate}
\item[(i)] If $\theta \neq \theta'$ are positive, then $\Gamma_p(\theta) \neq \Gamma_p(\theta')$;
\item[(ii)] For any $\theta \geq 0$, the curve is convex on the interval $[\theta,\theta+\pi]$.
\item[(iii)] For any line $L$ passing through $\Gamma_p(\theta')$ with $\theta < \theta' < \theta + \pi$, $L$ supports $\Gamma_p$ at $\Gamma_p(\theta')$ if and only if $L_{\theta'}$ is Birkhoff orthogonal to $L$.
\end{enumerate}
\end{lemma}

\begin{proof}
For any oriented tangent line $L_{\theta}$, let $L_{\theta}^+$ denote the `positive' half of the line; that is, the open half line obtained as the component of $L_{\theta} \setminus C$ whose points are at positives distance from any point of $L_{\theta} \cap C$ according to our orientation. Note that for any $\theta, \theta' \geq 0$, either $L_{\theta}^+ \cap L_{\theta'}^+ = \emptyset$ or $L_{\theta}^+ = L_{\theta'}^+$, and the latter case holds if and only if $\theta - \theta'$ is an integer multiple of $2\pi$. In both cases, if $\theta \neq \theta'$, then $\Gamma_p(\theta) \neq \Gamma_p(\theta')$ clearly holds. 

Let $\alpha = \max \{ 0, \theta-\pi \}$, and let $L$ be Birkhoff orthogonal to $L_{\theta}$ such that $\Gamma_p(\theta) \in L$. We show that $L$ supports the curve $\Gamma_p(t), t \in [\alpha, \theta+\pi]$, which shows both (ii), and the `if' part of (iii).
To do it, let $c$ be a point of $L_{\theta} \cap C$, and let $H$ denote the closed half plane bounded by $L$ and containing $c$.

Consider some $\alpha \leq t < \theta$. By the triangle inequality and the definition of $\Gamma_p$, the normed distance of $\Gamma_p(t)$ and $c$ is at most $d(p,c)$. Thus, by the triangle inequality, $\Gamma_p(t) \in c + d(p,c) M$. On the other hand, since $L$ is Birkhoff orthogonal to $L_{\theta}$, $L$ is a supporting line of $c + d(p,c) M$, implying that $c + d(p,c) M \subset H$. This shows that $\Gamma_p(t) \in H$ (see Figure~\ref{fig:convex_A}).

\begin{figure}[ht]
\begin{center}
\includegraphics[width=0.8\textwidth]{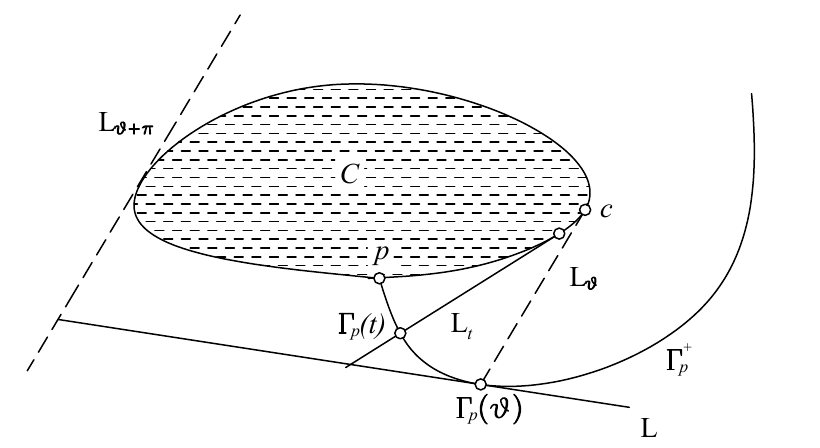}
\caption{An illustration for Lemma~\ref{lem:convexity} in the case $\alpha \leq t < \theta$.}
  \label{fig:convex_A}
\end{center}
\end{figure}

Consider some $\theta < t \leq \theta + \pi$.
Let $d$ be a point of $\Gamma_p(t) \cap C$. Let $b=\Gamma_p(\theta)$. Let $a$ be the intersection of $L$ and the half line $L_{t}^-$, if it exists, and let $e$ denote the intersection of $L_{\theta}$ and $L_{t}^-$, if it exists. Observe that if $a$ does not exist, then $L_t^- \subset H$, implying that $\Gamma_p(t) \in H$. Thus, we may assume that $a$ exists, which yields that also $e$ exists (see Figure~\ref{fig:convex_B}). Let $G(c,d)$ denote the arc of $\bd(C)$ from $c$ to $d$ according to the orientation of $\bd(C)$.
Now, since $L$ is Birkhoff orthogonal to $L_\theta$, $||e-b||_M \leq ||e-a||_M$. Hence, $||d-a||_M \geq ||e-b||_M + ||d-e||_M = ||c-b||_M + ||e-c||_M + ||d-e||_M$. By convexity, $G(c,d) \leq ||e-c||_M + ||d-e||_M$. Combining these inequalities we obtain that
\[
||d-a||_M \geq ||c-b||_M + G(c,d) = ||\Gamma_p(t)-d||_M.
\]
This implies that $\Gamma_p(t) \in [d,a] \subset H$.

\begin{figure}[ht]
\begin{center}
\includegraphics[width=0.8\textwidth]{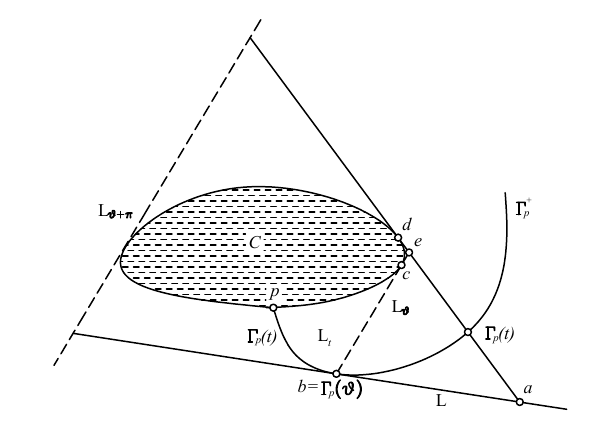}
\caption{An illustration for Lemma~\ref{lem:convexity} in the case $\theta < t < \theta + \pi$.}
  \label{fig:convex_B}
\end{center}
\end{figure}

Finally, let $L$ be a supporting line of $\Gamma_p$ at $\Gamma_p(\theta)$. For contradiction, suppose that $L_{\theta}$ is not Birkhoff orthogonal to $L$. Then there is some $\theta' \in (\alpha, \theta+\pi)$ such that $L_{\theta'}$ is Birkhoff orthogonal to a line parallel to $L$. On the other hand, by the convexity on $\Gamma_p$, this line must be $L$, and the segment connecting $\Gamma_p(\theta)$ and $\Gamma_p(\theta')$ must belong to $\Gamma_p$. Thus, for any relative interior point $\Gamma_p(t)$ of this segment, $L$ is the only supporting line of $\Gamma_p$ at this point, implying that $L_t$ is Birkhoff orthogonal to $L$. By compactness, from this it follows that $L_{\theta}$ is Birkhoff orthogonal to $L$, contradicting our assumption.
\end{proof}

\begin{remark}\label{rem:strictconvexity}
We note that if the norm is strictly convex, $\Gamma_p$ is strictly convex for any $\theta \geq 0$ on the interval $[\theta,\theta+\pi]$. Indeed, to prove it it is sufficient to observe that the strict convexity of the norm yields that for any supporting line $L$ of $\Gamma_p$ at some $\Gamma_p(\theta')$, there is a unique direction that is Birkhoff orthogonal to $L$ and thus, supporting lines at different points of $\Gamma_p$ have different directions.
\end{remark}

In the remaining part of the paper, following \cite{Rot94}, if, for a curve $\Gamma : [a,b] \to \mathbb{M}$ and point $p \in \mathbb{M}$, $t_1 \leq t_2$ implies $||\Gamma(t_1)-p||_M \leq ||\Gamma(t_2)-p||_M$, we say that \emph{$\Gamma$ satisfies the increasing chord property with respect to $p$}. Furthermore, if $\Gamma$ satisfies the increasing chord property with respect to any $p \in S$ for some nonempty set $S \subseteq \mathbb{M}$, we say that \emph{$\Gamma$ satisfies the increasing chord property with respect to $S$}.

\begin{lemma}\label{lem:monotonicity}
For any $\theta \geq 0$, the involute $\Gamma_p$ satisfies the increasing chord property on the interval $[\theta,\theta+\pi]$. Furthermore, $\Gamma_p([0,\pi])$ satisfies the increasing chord property with respect to $C$.
\end{lemma}

\begin{proof}
 Note that if $C$ is fixed, for any $\theta\in \Re$, $\Gamma_p(\theta)$ is a continuous function of the norm $\mathcal{M}$. Thus, as (the pointwise) limit of a sequence of curves satisfying the increasing chord property also satisfies this property, we may assume that $\mathcal{M}$ is smooth and strictly convex. Then, for any $\tau \in (\theta,\theta+\pi)$, there is a unique supporting line $S_\tau$ of $\Gamma_p$.

Consider any $\theta < \tau < \theta+\pi$. By the definition of $\Gamma_p$, $L_{\tau}$ separates $\Gamma_p([\theta,\tau])$ and $\Gamma_p([\tau,\theta+\pi])$. Thus, $[\Gamma_p(\tau), \Gamma_p(\tau')]$ is not parallel to $L_{\tau}$ for any $\tau' \in [\theta,\theta+\pi]$. This yields that there is an $\varepsilon > 0$ such that for any $\tau' \in (\tau-\varepsilon, \tau + \varepsilon)$ and $\tau'' \in [\theta,\theta+\pi]$, $\tau' \in [\tau,\tau'']$ implies $||\Gamma_p(\tau') - \Gamma_p(\tau'')||_M \leq ||\Gamma_p(\tau'') - \Gamma_p(\tau)||_M$, and $\tau \in [\tau',\tau'']$ implies $||\Gamma_p(\tau') - \Gamma_p(\tau'')||_M \geq ||\Gamma_p(\tau') - \Gamma_p(\tau)||_M$. In other words, for any fixed value $\tau$, $\Gamma_p$ `locally' satisfies the increasing chord property at $\Gamma_p(\tau)$. But since $\Gamma_p$ is compact on the given interval, this yields that $\Gamma_p([\theta,\theta+\pi])$ `globally' satisfies this property.

Now we show that $\Gamma_p([0,\pi])$ satisfies the increasing chord property with respect to $C$. Again, by continuity, we may assume that both $C$ and $M$ are smooth and strictly convex. Suppose for contradiction that there is some point $x \in C$ such that $||\Gamma_p(\tau) - x||_M > ||\Gamma_p(\tau') - x||_M$ for some $0 \leq \tau < \tau' \leq \pi$. According to our conditions, $\Gamma_p$ is differentiable, and thus it follows that there is some $\bar{\tau} \in (\tau,\tau')$ such that $\left. ||\Gamma_p(\sigma) - x||'_{M}\right|_{\sigma=\bar{\tau}} < 0$. This implies that the line $L_{\bar{\tau}}$ through $\Gamma_p(\bar{\tau})$ and Birkhoff orthogonal to $\Gamma'_p(\bar{\tau})$ does not separate $C$ from $\Gamma_p([\bar{\tau},\pi])$, leading to a contradiction with the definition of $\Gamma_p$.
\end{proof}

\section{The proof of Theorem~\ref{thm:main}}\label{sec:main}

Recall that, according to our conditions, $M$ is strictly convex. This yields, in particular, that the bisector of any segment is a continuous curve which intersects any line parallel to the segment in a single point. From this it also follows that the complement of any such bisector in $\mathbb{M}$ is the disjoint union of two connected and simply connected components.

Since applying any affine transformation to both $f$ and $M$ simultaneously does not change the $M$-length of $f$, we may assume (throughout this section) that $f(0)=(0,0)$, $f(1)=(1,0)$, and the lines $x=1$, $x=-1$ support $M$ at $(1,0)$ and $(-1,0)$, respectively.

We start the proof with Lemma~\ref{lem:xmonotone}, which, for the Euclidean case, coincides with \cite[Lemma 1]{Rot94}.

\begin{lemma}\label{lem:xmonotone}
The curve $f$ is strictly monotone in the $x$-direction, i.e. if $f(t) = (x(t),y(t))$, then $0 \leq t_1 < t_2 \leq 1$ implies $x(t_1) < x(t_2)$.
\end{lemma}

\begin{proof}
Consider the point $f(t)=(x(t),y(t))$, where $0 < t < 1$. Since $f(t) \in (f(0)+M) \cap (f(1)+M)$, we have that $0 < x(t) < 1$. Let $||f(t)-f(0)||_M=d_0$, $||f(t)-f(1)||_M=d_1$. Then, for any $t < t' \leq 1$, $f(t') \in (f(1)+d_1 M) \setminus \inter (d_0 M)$. But $f(t)$ has a neighborhood $U$ such that for any $q \in U \setminus \{ f(t) \}$ in this intersection, the $x$-coordinate of $q$ is strictly larger than $x(t)$. Thus, there is some $t < s$ such that $x(t) < x(t')$ for any $t' \in (t,s]$. Hence the fact that the interval $[0,1]$ is connected implies the assertion.
\end{proof}

\begin{remark}\label{rem:bisector}
By the definition of the increasing chord property, the fact that $f$ satisfies this property is equivalent to the property that for any $0 \leq t_1 < t_2 \leq 1$, the bisector of the segment $[f(t_1),f(t_2)]$ separates $f([0,t_1])$ and $f([t_2,1])$.
\end{remark}

Remark~\ref{rem:bisector} yields two useful corollaries. To state them, 
we define the following: if for sequences $\{t_n \}$, $\{t'_n\}$ of parameters we have that $t_n, t_n' \to t$, and for some unit vector $v \in \bd (M)$, $\frac{f(t_n)-f(t'_n)}{||f(t_n)-f(t'_n)||_M} \to v$, we say that $v$ is a \emph{(generalized) tangent vector} of $f$ at $f(t)$, and the line $L$ through $f(t)$ such that $v$ is Birkhoff orthogonal to $L$ is called a \emph{(generalized) normal line} of $f$ at $f(t)$.

\begin{corollary}\label{cor:orthogonality}
Any normal line $L$ of $f$ at a point $f(t)$ separates $f([0,t])$ and $f([t,1])$.
\end{corollary}

\begin{proof}
The statement follows Remark~\ref{rem:bisector}, the fact that $f([0,1]) \subset X=(p+M) \cap (q+M)$ is bounded, and that for any sequences $\{t_n \}, \{t_n' \}$ of parameters with the property that $t_n, t_n' \to t$, and $\frac{f(t_n)-f(t_n')}{||f(t_n)-f(t_n')||_M} \to v$, where $v$ is Birkhoff orthogonal to $L$, the intersection of the bisector of $[f(t_n),f(t_n')]$ converges to $L \cap X$ with respect to Hausdorff distance. 
\end{proof}

\begin{corollary}\label{cor:convexhull}
Let $S$ be a topological disk. If $f$ has increasing chords with respect to $\bd (S)$, then it has increasing chords with respect to $S$.
\end{corollary}

\begin{proof}
Consider any $0 \leq t_1 < t_2 \leq 1$, and let $B(t_1,t_2)$ denote the bisector of the segment $[f(t_1),f(t_2)]$. Then $f([0,t_1])$ is contained in one of the two connected components of $\mathcal{M} \setminus B(t_1,t_2)$. The property that $f$ has increasing chords with respect to $\bd (S)$ implies that $\bd(S)$ is contained in the same component for any $0 \leq t_1 < t_2 \leq 1$. But then, since both $S$ and the two components of $\mathcal{M} \setminus B(t_1,t_2)$ are simply connected, we have that $S$ is contained in the same component as well.
\end{proof}

Let $0 < t_1 < t_2 < \ldots < t_k < 1$ be an arbitrary subdivision of $[0,1]$. Let $a_i=f(t_i)$ for all values of $i$. We now create a new curve $g_i$ from $a_i$ to $a_{i+1}$. We want $g_i$ to be the shortest curve with the increasing chord property with respect to $f([0,t_i])$. 
To find this curve, we first prove Lemma~\ref{lem:separates}.

\begin{lemma}\label{lem:separates}
The curve $f([t_i,t_{i+1}])$ separates the arcs of the involutes of $\conv \{ f([0,t_i]) \}$ and $\conv \{ f([t_{i+1},1]) \}$ at $f(t_i)$ and $f(t_{i+1})$, respectively, with parameter values $\theta \in (-\pi,\pi))$.
\end{lemma}

\begin{proof}
Set $K=\conv \{ f([0,t_i]) \}$, and observe that for any supporting line $L$ of $L$, the endpoints of $K \cap L$ belong to $f$. Let $f(\bar{t})$ be the point of $\bd (K)$ with the smallest parameter value. Note that by the connectedness of $f([0,t_i])$, the points $f(\bar{t})$ and $f(t_i)$ decompose $\bd (K)$ into two arcs on which the parameters of the points of $f$ are increasing from $f(\bar{t})$ towards $a_i=f(t_i)$. Let $\Gamma_0$ be the involute of $K$ through $a_i$. By Lemma~\ref{lem:convexity}, the curve $\Gamma=\Gamma_0([0,\pi])$ is convex.

We show that $f([t_i,t_{i+1}])$ is disjoint from the interior $G$ of the region bounded by $\Gamma$, the tangent line of $K$ Birkhoff orthogonal to $\Gamma$ and $\Gamma_0(\pi)$, and the corresponding arc of $\bd(K)$. 
Suppose for contradiction that $q=f(t^*)$, where $t_i < t^* < t_{i+1}$, is contained in this set. Note that $G$ can be decomposed into a $1$-parameter family of involutes $\Gamma_s$, where $s$ denotes the normed length of the arc between the base point $a_i(s)$ of the involute and $a_i$ on $\bd(K)$. Assume that $q \in \Gamma_{s^*}$. For simplicity, we set $q_0=q$.
Now, for any $\delta > 0$, we describe the following process: Let $q_1$ be the intersection point of $\Gamma_{s^*-\delta}$ and the tangent half line of $\Gamma_{s^*}$ at $q_0$ towards the base point $a_i(s^*)$ of $\Gamma_{s^*}$. Then let $q_2$ be the intersection point of $\Gamma_{s^*-2\delta}$ and the tangent half line of $\Gamma_{s^*-\delta}$ at $q_1$ towards $a_i(s^*-\delta)$. In general, let $q_{j+1}$ be the intersection point of $\Gamma_{s^*-(j+1)\delta}$ and the tangent half line of $\Gamma_{s^*-j\delta}$ at $q_j$ towards $a_i(s^*-j \delta)$. The process terminates when the segment reaches $\bd(K)$ at a point $\bar{q}(\delta)$. Clearly, as $\delta \to 0$, $\bar{q}(\delta) \to a_i(s^*)$. Let us choose a sufficiently small value of $\delta$ such that this base points lies in the interior of the arc of $\bd(K)$ between $a_i$ and $a_i(s^*)$.

\begin{figure}[ht]
\begin{center}
\includegraphics[width=0.8\textwidth]{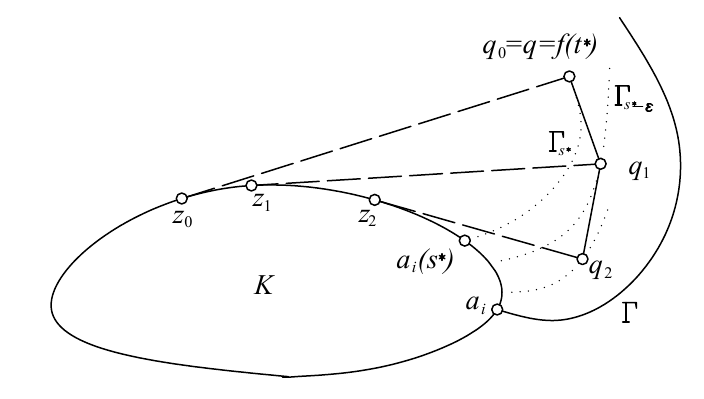}
\caption{An illustration for the proof of Lemma~\ref{lem:separates}.}
  \label{fig:notinside}
\end{center}
\end{figure}

For every $q_j$, let $z_j$ denote the tangent point of the tangent line of $\bd(K)$ through $q_j$; if this tangent line intersects $\bd(K)$ in a segment, $z_j$ denotes any of its endpoints (see Figure~\ref{fig:notinside}). Let $R_j$ denote the convex angular region bounded by the two half lines, starting at $z_j$, and passing through $q_j$ and $q_{j+1}$, respectively, and let $R$ denote the union of these regions. Note that every $z_j$ belongs to $f([0,t_i])$ by the definition of $K$ and our choice of the $z_j$. Thus, as $f$ satisfies the increasing chord property, every point of $f([t_i,t^*])$ in $R_0$ is contained in the triangle $\conv \{ z_0, q_0, q_1 \}$. Similarly, we have that the same statement is true for $R_1$ and the triangle $\conv \{ z_1, q_1, q_2 \}$. Continuing this reasoning, and using the fact that $z_{j+1} \in \conv \{ z_j,q_j,q_{j+1} \}$, we obtain that every point of $f([t_i,t^*])$ in $R$ is contained in the convex hull of the points $z_0, q_0, q_1, \ldots$.
This contradicts the assumption that $a_i$ is a point of $f([t_i,t^*])$. Thus, $f([t_i,t_{i+1}])$ does not intersect the interior of $G$.

Applying the same argument for the involute of $\conv f([t_{i+1},1])$ at $f(t_{i+1})$, we obtain that $f([t_i,t_{i+1}])$ separates these two involutes, implying also that they do not overlap.
\end{proof}

Note that $f([t_i,t_{i+1}])$ has increasing chords with respect to $f([0,t_i])$ and that Lemma~\ref{lem:separates} applies to it. We construct $g_i$ as follows. First we take the convex hull $K$ of $f([0,t_i])$ and the involute with the starting point $a_i$, on the other side we take the convex hull of $f([t_{i+1},1])$ and the involute with the starting point of $a_{i+1}$. Now $g_i$ is allowed to run in the area bounded by these involutes and $g_i$ is the shortest curve from $a_i$ to $a_{i+1}$ satisfying this condition. Generalizing \cite[Lemma 3]{Rot94} for strictly convex norms, we prove Lemma~\ref{lem:main}.

\begin{lemma}\label{lem:main}
We have the following.
\begin{enumerate}
\item[(i)] The typical shape of $g_i$ consists of
\begin{enumerate}
\item[(A)] an initial piece of one of the involutes starting at $a_i$,
\item[(B)] a straight line segment, and
\item[(C)] an initial piece of one of the involutes starting at $a_{i+1}$.
\end{enumerate}
Any of (A), (B) or (C) may be missing. If (A) and (B)exists, then (B) is tangent to (A), and the same holds for (B) and (C) if they exist. If both (A) and (C) exist, they turn in opposite directions.
\item[(ii)] If (A) exists, letting $u$ be the tangent vector of $K$ at $a_i$ in the direction on $\bd(K)$ belonging to (A), and letting $z$ be the (unique) point of the involute with a supporting line parallel to $u$, (A) is contained in the arc of the involute connecting $a_i$ and $z$. Similar property holds also for (C) if it exists. 
\item[(iii)] The curve $g_i$ has increasing chords with respect to $f([0,t_i])$.
\item[(iv)] The inverse curve of $g_i$ has increasing chords with respect to $f([t_{i+1},1])$.
\item[(v)] The curve $g_i$ has increasing chords.
\item[(vi)] The curve $g_i$ is not longer than $f([t_i,t_{i+1}])$.
\item[(vii)] $g_i \subseteq \conv \{ f([t_i,t_{i+1}]) \}$.
\end{enumerate}
\end{lemma}

\begin{proof}
The statement in (i) follows from Lemma~\ref{lem:separates}.

We prove (ii) by contradiction. Let $L$ be the supporting line of $K$ at $a_i$ with tangent vector $u$, and let $L_z$ denote the supporting line of $K$ containing $z$; here $L_z$ is Birkhoff orthogonal to the supporting line of the involute at $z$. By our assumption, $z$ belongs to (A). Let $y$ be a point of $f([0,t_i])$ on $L_z$. We either have a point $f(s)$, with $0 \leq s < t_i$ on the intersection of $L$ with the part of $\bd(K)$ between $a_i$ and $y$, or there is a sequence of points $f(s_j)$ on this part of $\bd(K)$ such that the direction of the segment $[f(s_j), f(t_i)]$ approaches that of $L$. Note that if $d=||z-a_i||_M$, then $f(s) \in z + d \inter (M)$ in the first case, and in the second case the same holds for $f(s_j)$ if $j$ is sufficiently large. This yields that $||f(s)-z||_M < d$ in the first case and, for large $j$, $||f(s_j)-z||_M < d$ in the second case. Since $z$ is a point of (A), $f([t_i,t_{i+1}])$ intersects the half line of $L_z$, starting at $y$ and containing $z$, at a point $\bar{z}=f(\bar{t})$, where $t_i < \bar{t} < t_{i+1}$. Then we have $z \in [y,\bar{z}]$ (see Lemma~\ref{lem:separates}), and since $L_z$ is Birkhoff orthogonal to $L$ this implies that $||f(s)-\bar{z}||_M < ||a_i-\bar{z}||_M$ in the first case, and $||f(s_j)-\bar{z}||_M < ||a_i-\bar{z}||_M$ in the second case, contradicting our assumption that $f$ satisfies the increasing chord property.

Next, we prove (iii). To show that (A) satisfies the increasing chord property with respect to $K$, we can apply Lemma~\ref{lem:monotonicity}. The same statement for (B) follows from the fact that the segment (B) is tangent to (A). To show that (C) also satisfies this property, note that if there is a sequence of convex disks $K_n$ and curves $c_n$ satisfying this property such that $K_n \to K$ and $c_n \to c$, then $c$ also have the increasing chord property to $K$. Thus, we may assume that (C) is a differentiable curve. Then, using the idea in the second half of the proof of Lemma~\ref{lem:monotonicity}, it is sufficient to check if the normal lines of (C) separate $K$ from the part of the involute ending at $a_{i+1}$. On the other hand, such a normal line $L$ is tangent to $K'=\conv (f([t_{i+1},1]))$, and thus, like in the proof of \cite[Lemma 3]{Rot94}, there exist points $f(s_1), f(s_2), f(s_3)$ on $L$ such that $0 \leq s_1 \leq t_i, t_i < s_2 < t_{i+1}, t_{i+1} \leq s_3 \leq 1$, satisfying $f(s_3) \in [f(s_2),f(s_1)]$, which contradicts the assumption that $f$ satisfies the increasing chord property.

The statement in (iv) follows in the same way as (iii). The statement in (v) can be proved using Lemma~\ref{lem:monotonicity}, and the statements in (i) and (ii).
The statement in (vi) follows from the definition of $g_i$, and the statement in (vii) can be obtained from Lemma~\ref{lem:separates} by considering the possible paths connecting $a_i$ to $a_{i+1}$.
\end{proof}

Based on Lemma~\ref{lem:main}, carrying out the procedure described above and gluing together the pieces $g_i$ we obtain 
a curve $g$ satisfying the following:
\begin{itemize}
\item[(a)] it satisfies the increasing chord property;
\item[(b)] $\arclength_M(f)-\varepsilon \leq \arclength_M(g) \leq \arclength_M(f)$ (see (vi) of Lemma~\ref{lem:main});
\item[(c)] it is the union of finitely many strictly convex and strictly concave arcs (as functions of $x$), and linear segments (see Remark~\ref{rem:strictconvexity}).
\end{itemize}

From now on we will work with the curve $g$. In the next part, we `convexify' the curve $g:[0,1]\xrightarrow{}\mathbb{R}^2$ in the following way: we cut $g$ into infinitesimal pieces and rearrange them according to slope. Since $g$ consists of finitely many convex parts, we do not need to worry about problems with limits and give an explicit construction for the convexified curve $g'$.

With a little abuse of notation, let $0 = t_0 < t_1 < t_2 < \ldots < t_{m-1} < t_m = 1$ be a subdivision of $[0,1]$ such
that each piece $g([t_{i-1},t_i])$ is either a straight line segment, or strictly convex, or strictly concave. To define the convexified curve $g'$ we need to re-parameterize $g$. Our parameter will be essentially the angle of the tangent line at the point, but because of the straight line segments we introduce a more complex parameterization.
Namely, the convexified curve $g': [0,U]\xrightarrow{}\mathbb{R}^2$ will be parameterized by a parameter $u$ as follows.

For every value of $i$ and parameter $u$, we choose an angle $\alpha(u)$ in the range $\left( -\frac{\pi}{2}, \frac{\pi}{2} \right)$ and  two points $r_i(u)$ and $s_i(u)$ from the interval $[t_{i-1},t_i]$.
If $g([t_i,t_{i+1}])$ is not a straight line segment, we set $\alpha(u) = u$, and choose these values such that, when measured by angle with the positive $x$-axis, for every $i$, $g([r_i(u), s_i(u)])$ consists of the points of $g([t_{i-1},t_i])$ where the angle of $g$ is not less than $\alpha(u)$. This yields, in particular, that if $g([t_{i-1},t_i])$ is strictly concave, then $r_i(u) = t_{i-1}$. Furthermore, to define $s_i(u)$ observe that there is at most one parameter value $t_{i-1} \leq t \leq t_i$ such that $g(t)$ lies on the supporting line of $g([t_{i-1},t_i])$ with angle $\alpha(u)$. If there is such a value $t$, we choose $s_i(u)$ as this value, and in the opposite case we set $s_i(u)=t_i$. Similarly, if $g([t_{i-1},t_i])$ is strictly convex, we set $s_i(u)=t_i$, and define $r_i(u)$ either as $t_{i-1}$, or the unique value $t \in [t_{i-1},t_i]$ such that $g(t)$ lies on the supporting line of $g([t_{i-1},t_i])$ with angle $\alpha(u)$.

If $g([t_{i-1},t_i])$ is a straight line segment, we set $r_i(u)=s_i(u)=t_i$ as long as $\alpha(u)$ is steeper than the segment. Then there is an interval $[u,u']$ during which $\alpha(u)$ remains stationary and parallel to the segment. In this interval $r_i(u)$ decreases to $t_{i-1}$. For higher values of $u$, $r_i(u)$ and $s_i(u)$ remain constant again.

Clearly, $\alpha(u)$ is decreasing. It is easy to construct such parameterizations $\alpha(u)$, $r_i(u)$ and $s_i(u)$ with $u \in [0,U]$, where $\alpha(u)$ is continuous.

Now we define the curve $g':[0,U]\xrightarrow{}\mathbb{M}$ by the following equation:
\begin{equation}\label{rep:1}
   g'(u)=g'(u)-f(0)=\sum_{i=1}^m \left(g(s_i(u))-g(r_i(u)) \right). 
\end{equation}
Note that this is only a weakly monotonic parameterization of $g'$.

Clearly, $g'$ is a curve from $f(0)$ to $f(1)$. Furthermore, it is convex because at each point $g'(u)$ the line with direction $\alpha(u)$ is a supporting line. The curve $g'$ lies in the upper half-plane $\{ y \geq 0\}$, and moreover, it has the same length as $g$. Indeed, as $u$ increases by a small amount the parts of $g$ that are added to $g'$ are essentially parallel. More specifically, we have
\[
g'(u+\epsilon)-g'(u)=\sum_{i=1}^m \left( g(r_i(u))-g(r_i(u+\epsilon)) \right) +\sum_{i=1}^m \left( g(s_i(u+\epsilon)) - g(s_i(u)) \right).\]
The directions of the non-zero differences of vectors in these sums are in the range $[\alpha(u+\epsilon), \alpha(u)]$. Hence, the length of the vector sum differs from the sum of the lengths of the vectors by at most some factor $(1+\delta(\varepsilon)) \cos{(\alpha(u)-\alpha(u+\epsilon))}$. Here the constant
$\delta(\varepsilon)$ is defined as the smallest value such that the (Euclidean) radial function $\rho_M : \Sph^1 \to \Re$ of the unit disk $M$ satisfies the inequality $\left| \rho_M(\beta-\rho_M((\alpha(u)) \right| \leq \delta(\varepsilon)$ for any $\beta \in [\alpha(u-\epsilon), \alpha(u+\varepsilon)]$. For this constant we clearly have $\delta(\varepsilon) \to 0$ as $\varepsilon \to 0^+$. Thus, the constant $(1+\delta(\varepsilon)) \cos{(\alpha(u)-\alpha(u+\epsilon))}$ can be made arbitrarily close to $1$.

\begin{figure}[ht]
\begin{center}
\includegraphics[width=0.45\textwidth]{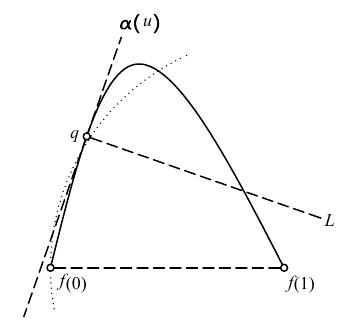}
\caption{The point $q$ on the curve $g'$.}
  \label{fig:Rote1}
\end{center}
\end{figure}

Now we show that each point of the curve $g'$ has an $M$-distance at  most $1$ from $f(0)$ and $f(1)$. Suppose for contradiction that e.g. the largest distance from $f(1)$ is greater than $1$.
Then there must be a point $q$ on $g'$ different from $f(0)$ with a supporting line of direction $\alpha$ such that the line $L$ which is Birkhoff orthogonal to the supporting line at the point $q$ passes above $f(1)$, i.e. it does not separate $f(0)$ and $f(1)$ (see Figure~\ref{fig:Rote1}). We take $\alpha$ in the direction of the curve $g'$ running from $f(0)$ to $f(1)$, and we find a parameter value $u$ such that $\alpha(u)$ is that direction. According to the representation in (\ref{rep:1}), we can cut the curve $g$ into parts at the points $r_j(u)$ and $s_j(u)$. We get two kinds of parts: The `upward' parts $g([r_j(u), s_j(u)])$, where the angle is always at least $\alpha$, and the remaining `downward' parts $g([s_{j-1}(u), r_j(u)])$, $g([0, s_1(u)])$ and $g([t_m(u), 1])$, where the angle is smaller. By discarding empty `parts' and merging together adjacent non-empty parts which are all upward or which are all downward we can assume that upward and downward parts alternate on the curve, and that the parts are proper parts of positive length. The sum of the upward vectors is $q=q-f(0)$, by \ref{rep:1}, and the downward vectors add up to $f(1)-q$. 

Since $f(1)$ is below the line $L$, there is a downward part $g([u,v])$ whose endpoint is below the line through its starting point parallel to $L$. This part must be adjacent to some upward part (see Figure~\ref{fig:Rote2}).
However, this leads to a contradiction with the increasing chord property of the involute: As one moves away from the common subdivision point on the upward part the distance to the other endpoint of the downward part decreases.

\begin{figure}[ht]
\begin{center}
\includegraphics[width=0.45\textwidth]{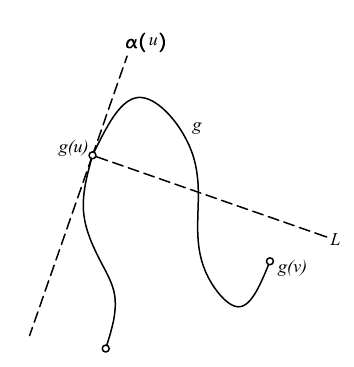}
\caption{The arc $g([u,v])$.}
  \label{fig:Rote2}
\end{center}
\end{figure}

Now we have shown that the convex curve $g'$ is contained in the convex disk $M \cap (f(1)+M)$. By symmetry, we also have that the reflection $g''$ of $g'$ to the midpoint of $[f(0),f(1)]$ is contained in $M \cap (f(1)+M)$. The curves $g'$ and $g''$ do not cross, as they are separated by the $x$-axis. It is a well known fact that for any two convex disks $K \subseteq L$ in a normed plane $\mathbb{M}$, the inequality $\perim_M(K) \leq \perim_M(L)$ holds (see e.g. \cite{BL24, Schaffer}). This yields that $\arclength_M(g') \leq \frac{1}{2} \perim_M(M \cap (f(1)+M))$, which is what we wanted to prove.

\section{The proof of Theorem~\ref{thm:extremals}}\label{sec:extremals}

We prove (i)-(iv) of Theorem~\ref{thm:extremals} in separate subsections.

\subsection{The proof of (i)}

Consider points $p,q \in \mathbb{M}$ with $||q-p||_M=1$. Let $X=(p+M) \cap (q+M)$. Since $\bd(X)$ is a connected curve that is the union of the compact sets $\bd(p+M) \cap \bd(X)$ and $\bd(q+M) \cap \bd(X)$, there is a point $x \in \bd(X)$ which belongs to the boundaries of both $p+M$ and $q+M$. Let the reflection of $x$ to the midpoint of $[p,q]$ be $x'$. Let $X'=\conv \{ x,p,q,x'\}$. Since $X' \subseteq X$ and both $X, X'$ are convex disks, we have $\perim_M(X') \leq \perim_M(X)$. On the other hand, by our choice of $x$, we have $\perim_M(X')=4$. This yields the inequality part. To show the equality part, note that if $M$ is a parallelogram or an affinely regular hexagon, and $q$ is a vertex of $p+M$, then $\perim_M\left( (p+M) \cap (q+M)\right)=2$.

\subsection{The proof of (ii)}\label{subsec:proof2}

We show that if $M$ is an affinely regular hexagon and $||q-p||_M=1$, then $\perim_M((p+M) \cap (q+M)) = 4$. This, combined with the statement in (i), yields (ii). To prove it we can assume that $M$ is a regular hexagon, since the statement remains the same if we apply an affine transformation on $M$.

Now, the proof is based on the connection of Reuleaux triangles in $\mathbb{M}$ and affinely regular hexagons inscribed in $M$. This connection is described in \cite{C66} (see also \cite{L16}). According to this, every Reuleaux triangle, in which the distance of the vertices is one, can be obtained in the following way: Choose an $o$-symmetric affinely regular hexagon $H$ inscribed in $M$. Let the vertices of $H$ be $p,q, r, -p, -q, -r$ (see Figure~\ref{fig:hexagon}). Then $r=q-p$, and the vertices decompose $\bd(M)$ into six arcs: $\Gamma_1, \Gamma_2, \Gamma_3$ and their reflections to $o$ (see Figure~\ref{fig:hexagon}). Choose any two consecutive vertices of $H$, say $p,q$. Then $R=M \cap (p+M) \cap (q+M)$ is a Reuleaux triangle in $\mathbb{M}$, and $\bd(R)$ can be decomposed into three arcs which are translates or reflections of $\Gamma_1, \Gamma_2, \Gamma_3$, respectively. In particular, we have $\perim_M(R)=\frac{1}{2} \perim_{M}(M)$.
We obtain similarly that if $X= (p+M) \cap (q+M)$, then $\perim_M(X)= \perim_M(M) - 2\arclength_M(\Gamma_i)$ for
some $i \in \{ 1,2,3 \}$.

\begin{figure}[ht]
\begin{center}
\includegraphics[width=0.5\textwidth]{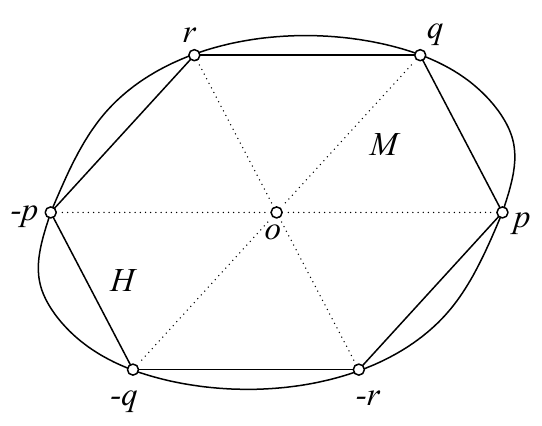}
\caption{An affinely regular hexagon $H$ inscribed in the $o$-symmetric convex disk $M$. The dotted diagonals of $H$ decompose it into six equilateral triangles of unit edgelength.}
  \label{fig:hexagon}
\end{center}
\end{figure}

To finish the proof, we recall the well-known fact that for every $p \in \bd(M)$ there is an affinely regular hexagon inscribed in $M$. If $\bd(M)$ does not contain two translates of the segment $(1+\varepsilon)[o,p]$ for some $\varepsilon > 0$, then this hexagon is unique, if it does, then every such hexagon has two opposite sides contained in $\bd(M)$. Consider a point $p \in \bd(M)$, and recall that $M$ is a regular hexagon. Let $H$ be the convex hull of $p$ and its rotated copies around $o$ by angles $\frac{2i\pi}{3}$ for $i=1,2,\ldots, 5$. Then $H$ is an (affinely) regular hexagon inscribed in $M$ and having $p$ as a vertex. This hexagon has no side contained in $M$, implying that this is the only such hexagon inscribed in $M$. Note that $\perim_M(M)=6$. Thus, by symmetry, the vertices of $H$ decompose $\bd(M)$ into six arcs of unit length. This yields that if $q,r$ are consecutive vertices of $H$, then $\perim_M((q+M) \cap (r+M))=4$, implying (ii).

\subsection{The proof of (iii)}\label{subsec:proof3}

The proof is based on the connection between Reuleaux triangles and affinely regular hexagons detailed in the proof of (ii).

For any $M \in \mathcal{M}_2$, let $\mathcal{H}_M$ denote the family of affinely regular hexagons inscribed in $M$.
Let $H \in \mathcal{H}_M$. Then, as we remarked in the previous subsection, the vertices of $H$ decompose $\bd(M)$ into six arcs, which in the following we denote by $\Gamma_1^H$, $\Gamma_2^H$, $\Gamma_3^H$, $-\Gamma_1^H$, $-\Gamma_2^H$ and $-\Gamma_3^H$.
Then, for any fixed $M \in \mathcal{M}$, we have:
\begin{multline*}
\min \{ L_M(p,q): ||p-q||_M=1 \} = \\
\frac{1}{2} \perim_M(M) - \max \{ \arclength_M(\Gamma_i^H): i=1,2,3 \hbox{ and } H \in \mathcal{H}_M \}.
\end{multline*}

It is well known that for any $M \in \mathcal{M}$, $\perim_M(M) \leq 8$ (see e.g. \cite{Golab}). On the other hand, we clearly have
$\max \{ \arclength_M(\Gamma_i^H): i=1,2,3 \} \geq \frac{1}{6} \perim_M(M)$. Thus,
\[
\min \{ L_M(p,q): ||p-q||_M=1 \} \leq \frac{1}{3} \perim_M(M) \leq \frac{8}{3}.
\]

For the other inequality in (iii), we observe that if $M$ is a Euclidean disk, then $\min \{ L_M(p,q): ||p-q||_M=1 \} = \frac{2\pi}{3}$.

\subsection{The proof of (iv)}

Let $(p+M) \cap (q+M) = X$. We need to show that $\perim_M \left( X \right) \leq 6$.
Let $L_p$, $L_q$ be parallel supporting lines of $X$ at $p$ and $q$, respectively, and let  $L_+$ and $L_-$ be the supporting lines of $X$ parallel to the line $L$ of $[p,q]$. Let the intersection points of $L_+$ with $L_p$ and $L_q$ be $x_p^+$ and $x_q^+$, respectively, and we define $x_p^-$ and $x_q^-$ similarly (see Figure~\ref{fig:MaxMax}). By the properties of supporting lines, $X \subseteq P=\conv \{ x_p^+, x_p^-, x_q^+,x_q^- \}$. Since $X$ is convex, this implies that $\perim_M(X) \leq \perim_M(P) = ||x_p^+-x_p^-||_M + ||x_q^+-x_q^-||_M + ||x_q^+-x_p^+||_M + ||x_q^--x_p^-||_M = 4 ||x_p^+ - p||_M + 2$. Thus, we need to show that $||x_p^+ - p||_M \leq 1$, or equivalently, $x_p^+ \in p+M$.

\begin{figure}[ht]
\begin{center}
\includegraphics[width=0.7\textwidth]{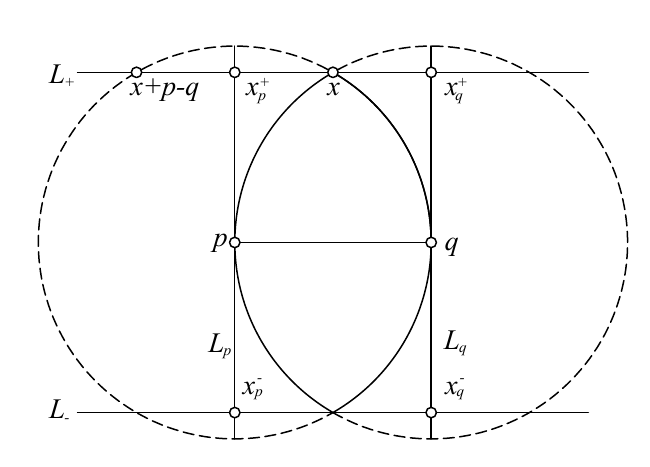}
\caption{An illustration for the proof of (iv) of Theorem~\ref{thm:extremals} with a Euclidean unit disk as $M$.}
  \label{fig:MaxMax}
\end{center}
\end{figure}

To show this, note that there is a point $x$ of $X=(p+M) \cap (q+M)$ on the segment $[x_p^+,x_q^+]$, and since $x \in q+M$, we have $x + (p-q) \in p + M$. But the distance of $x$ and $x+(p-q)$ is equal to $||p-q||_M = 1 = ||x_p^+-x_q^+||_M$, showing that $x+p-q$ is outside $P$. 

Thus, $x_p^+ \in [x,x+p-q] \subseteq p+M$. For the equality case, we observe that if $M$ is a parallelogram and $[p,q]$ is parallel to a side of $M$, then $\perim_M\left( (p+M) \cap (q+M)\right)=6$.

\section{Additional remarks and questions}\label{sec:remarks}

Not much is known about the maximum arclength of curves with the increasing chord property in higher dimensions even in the Euclidean case. The only published result in this direction is due to Rote \cite{Rot94}, who gave an upper bound of the arclengths these curves in $3$-dimensional Euclidean space. 
In addition, answering a question of Croft, Falconer and Guy \cite[G3]{CFG91} in the negative, he observed that the arc consisting of three consecutive edges of a Reuleaux tetrahedron does not satisfy the increasing chord property, but a slight modification of such a curve does.
In particular, nothing is known of the order of magnitude of the maximum arclength of these curves in $d$-dimensional Euclidean space, though the above example may indicate that the answer might be linear.
Our next example shows that the situation can be different among curves with this property in other normed spaces. 

\begin{proposition}\label{prop:highdim}
    Let $M_d$ be the cube $M_d=[-1,1]^d$. Then there is a curve $\Gamma_d$ in $\mathbb{M}_d$ with the increasing chord property with its endpoints unit distance apart and satisfying $\arclength_{M_d}(\Gamma_d) = 2^d -1$.
\end{proposition}

\begin{proof}
    We present a recursive construction of a curve $\Gamma_d$ for every dimension $d \geq 1$ that yields the previously stated value, and prove by induction on $d$ that $\Gamma_d$ preserves the increasing chord property for all $d \geq 1$.

    We begin with the base case. For $d=1$, i.e. the curve $\Gamma_1$ is a line segment of unit length, trivially satisfying the increasing chord property. The curve $\Gamma_2$ consists of three unit-length edges of the square $[0,1]^2$. This can be interpreted as follows: we identify $\Re$ with the $x$-axis, take two copies of $\Gamma_1$, place them one unit apart, and connecting them with a single edge of the square.

    The general construction proceeds analogously: given a curve $\Gamma_d$ in dimension $d$, we obtain $\Gamma_{d+1}$ as follows.
    We identify the coordinate hyperplane $\{ x_{d+1} = 0\}$ with $\Re^d$, consider the curves $\Gamma_d$ and $(0,0,\ldots, 0,1)+ \Gamma_d$, and connect an endpoint of $\Gamma_d$ with the corresponding endpoint of its translate. Note that the curves $\Gamma_d$ and $(0,0,\ldots, 0,1)+ \Gamma_d$ are in opposite facets of the cube $[0,1]^{d+1}$.

    The arclength of the curve $\Gamma_d$ satisfies the recurrence relation
    \[
    \arclength_{M_{d+1}} (\Gamma_{d+1}) = 2 \arclength_{M_{d}} (\Gamma_{d}) + 1, \quad \text{with } \arclength_{M_{1}} (\Gamma_{1}) = 1.\]
    Then, by induction we have $\arclength_{M_{d}} (\Gamma_{d}) = 2^d-1$.

    We now prove by induction on $d$ that $\Gamma_d$ satisfies the increasing chord property. In the cases $d=1$ and $d=2$ the statement is trivial. Now, assume that for some $d \geq 2$, $\Gamma_d$ satisfies this property. Consider the curve $\Gamma_{d+1}$. As described before, this curve is the union of $\Gamma_d$ and $\Gamma_d'=(0,0,\ldots, 0,1)+ \Gamma_d$, and an edge $[p,p']$ connecting an endpoint $p$ of $\Gamma_d$ and an endpoint $p'$ of $\Gamma_d'$.
    
By the induction hypothesis $\Gamma_d$ and its translate satisfy the increasing chord property. Furthermore, for any $q \in \Gamma_d$ and $q' \in \Gamma_d'$ we have $||q-q'||_{M_{d+1}}=1$. We note also that for any point $q$ in the hyperplane through $\Gamma_d$, the distance of $q$ from a point $z$ of $[p,p']$ does not decrease as $z$ moves on the segment from $p$ to $p'$; a similar statement holds also for a point of the hyperplane of $\Gamma_d'$. Thus, it remains to show that $\Gamma_d$ and $\Gamma_d'$ satisfy the increasing chord property with respect to any $z \in [p,p']$. To show it note that if $[p,w]$ is the edge of $\Gamma_d$ ending at $p$, then for any $z \in [p,p']$, the distance of $z$ from a point of $[p,w]$ does not decrease as the point moves from $p$ to $w$, and $||w-z||_{M_{d+1}}=1$. Furthermore, for any other point $u \in \Gamma_d$, we have $||z-u||_{M_{d+1}}=1$. A similar observation shows the statement form $\Gamma_d'$, finishing the proof of Proposition~\ref{prop:highdim}.
\end{proof}

In light of Theorem~\ref{thm:extremals}, it is a meaningful question to ask the following. 

\begin{problem}\label{prob:constant}
Find all $M \in \mathcal{M}$ such that the value of $L_M(p,q)$ is the same value for all $p,q \in \mathbb{M}$ with $||p-q||_M=1$. In particular, is it true that $M$ has this property if and only if an affine copy of $M$ has $6$-fold rotational symmetry?
\end{problem}

Regarding Problem~\ref{prob:constant}, we remark that the consideration in Subsection~\ref{subsec:proof2} shows that if $M$ has $6$-fold rotational symmetry, then the value of $L_M(p,q)$ is independent of the choice of the points $p,q$ (provided that they are unit distance apart).

\bigskip

\noindent
\textbf{Acknowledgements}

The authors express their gratitude to Adrian Dumitrescu for directing their attention to this interesting topic, and for the fruitful discussions on it.

\end{document}